\documentclass{amsart}
\usepackage{amsfonts,amssymb,amsmath,amsthm,mathrsfs}
\usepackage{url}
\usepackage{enumerate}

\newtheorem{theorem}{Theorem}[section]
\newtheorem{lemma}[theorem]{Lemma}

\newtheorem{corollary}[theorem]{Corollary}
\theoremstyle{definition}
\newtheorem{definition}[theorem]{Definition}
\newtheorem{remark}[theorem]{Remark}
\numberwithin{equation}{section}
\newtheorem{assumption}[theorem]{Assumption}

\begin{document}

\title[singular integral operators]{Weighted vector-valued bounds for the singular integral operators with nonsmooth kernels}

\author{Guoen Hu}
\address{Department of Applied Mathematics, Zhengzhou Information Science and Technology
Institute, Zhengzhou 450001, P. R. China}
\email{guoenxx@163.com}
\thanks{The research    was supported by
the NNSF of
China under grant $\#$11371370.}


\begin{abstract}Let $T$ be a singular integral operator with non-smooth kernel which was introduced by Duong and McIntosh. In this paper, we prove that this operator and  its corresponding grand maximal operator satisfy certain weak type endpoint vector-valued estimate of $L\log L$ type. As an application  we established a refined weighted vector-valued bound  for this operator.
\end{abstract}

\maketitle

\section{Introduction}
We will work on $\mathbb{R}^n$, $n\geq 1$. Let $A_p(\mathbb{R}^n)$ ($p\in (1,\,\infty)$) be the weight functions class of Muckenhoupt, that is, $w\in A_{p}(\mathbb{R}^n)$ if $w$ is nonnegative and locally integrable, and
$$[w]_{A_p}:=\sup_{Q}\Big(\frac{1}{|Q|}\int_Qw(x){\rm d}x\Big)\Big(\frac{1}{|Q|}\int_{Q}w^{-\frac{1}{p-1}}(x){\rm d}x\Big)^{p-1}<\infty,$$
where the  supremum is taken over all cubes in $\mathbb{R}^n$, $[w]_{A_p}$ is called the $A_p$ constant of $w$, see \cite{gra} for properties of $A_p(\mathbb{R}^n)$.
In the last several years, there has been significant progress in the study of sharp weighted bounds with $A_p$ weights for the classical operators in Harmonic Analysis. The study was begin by Buckley \cite{bu}, who proved that if $p\in (1,\,\infty)$ and $w\in A_{p}(\mathbb{R}^n)$, then the Hasrdy-Littlewood maximal operator $M$ satisfies
\begin{eqnarray}\|Mf\|_{L^{p}(\mathbb{R}^n,\,w)}\lesssim_{n,\,p}[w]_{A_p}^{\frac{1}{p-1}}\|f\|_{L^{p}(\mathbb{R}^n,\,w)},\end{eqnarray}
Moreover, the estimate (1.1) is sharp since the exponent $1/(p-1)$ can not be replaced by a smaller one. Hyt\"onen and P\'erez \cite{hp} improved the estimate (1.1), and showed that
\begin{eqnarray}\|Mf\|_{L^{p}(\mathbb{R}^n,\,w)}\lesssim_{n,\,p}\big([w]_{A_p}[w^{-\frac{1}{p-1}}]_{A_{\infty}}\big)^{\frac{1}{p}}\|f\|_{L^{p}(\mathbb{R}^n,\,w)}.\end{eqnarray}
where and in the following, for a weight $u$, $[u]_{A_{\infty}}$ is defined by
$$[u]_{A_{\infty}}=\sup_{Q\subset \mathbb{R}^n}\frac{1}{u(Q)}\int_{Q}M(u\chi_Q)(x){\rm d}x.$$
It is well known that for $w\in A_p(\mathbb{R}^n)$, $[w^{-\frac{1}{p-1}}]_{A_{\infty}}\lesssim [w]_{A_p}^{\frac{1}{p-1}}$. Thus, (1.2) is more subtle than (1.1).

The sharp dependence of the weighted estimates of  singular integral operators in terms of the $A_p(\mathbb{R}^n)$ constant  was first considered by Petermichl \cite{pet1,pet2}, who solved this question for Hilbert transform and Riesz transform.   Hyt\"onen \cite{hyt}
proved that  for a  Calder\'on-Zygmund operator $T$ and $w\in A_2(\mathbb{R}^n)$,
\begin{eqnarray}\|Tf\|_{L^{2}(\mathbb{R}^n,\,w)}\lesssim_{n}[w]_{A_2}\|f\|_{L^{2}(\mathbb{R}^n,\,w)}.\end{eqnarray}
This solved  the so-called $A_2$ conjecture. Combining the estimate (1.3) and the extrapolation theorem in \cite{dra}, we know that
for a Calder\'on-Zygmund operator $T$, $p\in (1,\,\infty)$ and $w\in A_p(\mathbb{R}^n)$,
\begin{eqnarray}\|Tf\|_{L^{p}(\mathbb{R}^n,\,w)}\lesssim_{n,\,p}[w]_{A_p}^{\max\{1,\,\frac{1}{p-1}\}}\|f\|_{L^{p}(\mathbb{R}^n,\,w)}.\end{eqnarray}
In \cite{ler2}, Lerner  gave a much simple proof of (1.4) by  controlling the Calder\'on-Zygmund operator using sparse operators.
Lerner \cite{ler2} proved that
\begin{theorem}\label{t1.1} Let $T$ be a sublinear operator and $\mathcal{M}_T$ be the corresponding grand maximal operator defined by
$$\mathcal{M}_Tf(x)=\sup_{Q\ni x}{\rm ess}\sup_{\xi\in Q}|T(f\chi_{\mathbb{R}^n\backslash 3Q})(\xi)|.
$$
Suppose that both $T$ and $\mathcal{M}_T$ are bounded from $L^1(\mathbb{R}^n)$ to $L^{1,\,\infty}(\mathbb{R}^n)$. Then  for $p\in (1,\,\infty)$ and $w\in A_p(\mathbb{R}^n)$, $T$ satisfies (1.4).
\end{theorem}

Let $p,\,r\in(0,\,\infty]$ and $w$ be a weight. As usual, for a sequence of numbers $\{a_k\}_{k=1}^{\infty}$, we denote $\|\{a_k\}\|_{l^r}=\big(\sum_k|a_k|^r\big)^{1/r}$. The space $L^p(l^{r};\,\mathbb{R}^n,\,w)$ is defined as
$$L^p(l^{r};\,\mathbb{R}^n,\,w)=\big\{\{f_k\}_{k=1}^{\infty}:\, \|\{f_k\}\|_{L^p(l^r;\,\mathbb{R}^n,\,w)}<\infty\big\}$$
where
$$\|\{f_k\}\|_{L^p(l^r;\,\mathbb{R}^n,\,w)}=\Big(\int_{\mathbb{R}^n}\|\{f_k(x)\}\|_{l^r}^pw(x)\,{\rm d}x\Big)^{1/p}.$$
When $w\equiv 1$, we denote  $\|\{f_k\}\|_{L^p(l^r;\,\mathbb{R}^n,\,w)}$
by $\|\{f_k\}\|_{L^p(l^r;\,\mathbb{R}^n)}$ for simplicity. Hu \cite{hu} extended Lerner's result to the  vector-valued case,  proved that
\begin{theorem}\label{t1.2} Let $T$ be a sublinear operator and $\mathcal{M}_T$ be the corresponding grand maximal operator.
Suppose that for some $q\in (1,\,\infty)$, \begin{eqnarray*}&&
\big|\big\{y\in \mathbb{R}^n:\,\|\{Tf_k(y)\}\|_{l^q}+\|\{\mathcal{M}_{T}f_k(y)\}\|_{l^q}>\lambda\big\}\big|\\
&&\quad\lesssim\int_{\mathbb{R}^n}\frac{\|\{f_k(y)\}\|_{l^q}}{\lambda}\log \Big(1+\frac{\|\{f_k(y)\}\|_{l^q}}{\lambda}\Big)dy.\nonumber\end{eqnarray*}Then for
all $p\in (1,\,\infty)$ and $w\in A_p(\mathbb{R}^n)$, $$\big\|\{Tf_k\}\big\|_{L^p(l^q;\mathbb{R}^n,w)}\lesssim_{n,\,p} [w]_{A_p}^{\frac{1}{p}}\big([w^{-\frac{1}{p-1}}]_{A_{\infty}}^{\frac{1}{p}}+[w]_{A_{\infty}}^{\frac{1}{p'}}\big)[w^{-\frac{1}{p-1}}]_{A_{\infty}}^{\beta}
\|\{f_k\}\|_{L^p(l^q;\,\mathbb{R}^n,\,w)}.$$
\end{theorem}

Let $T$ be a $L^2(\mathbb{R}^n)$ bounded linear operator with kernel $K$ in the sense that
for all $f\in L^2(\mathbb{R}^n)$ with compact support and a. e. $x\in\mathbb{R}^n\backslash {\rm supp}\, f$,
\begin{eqnarray}Tf(x)=\int_{\mathbb{R}^n}K(x,\,y)f(y)dy.\end{eqnarray}
where $K$ is a measurable function on $\mathbb{R}^n\times \mathbb{R}^n\backslash\{(x,\,y):\,x=y\}$. To obtain a weak
$(1,\,1)$ estimate for certain Riesz transforms, and $L^p$ boundedness with $p\in (1,\,\infty)$ of
holomorphic functional calculi of linear elliptic operators on irregular domains, Duong
and McIntosh \cite{duongmc} introduced singular integral operators with nonsmooth kernels on
spaces of homogeneous type via the following generalized approximation to the
identity.

\begin{definition}
A family of operators $\{A_t\}_{t>0}$ is said to be an approximation to the identity, if for
every $t>0$, $A_t$ can be represented by the kernel at in the following sense: for every function $u\in L^p(\mathbb{R}^n)$
with $p\in [1,\,\infty]$ and almost everywhere $x\in\mathbb{R}^n$,
$$A_tu(x)=\int_{\mathbb{R}^n}a_t(x,\,y)u(y)dy,$$
and the kernel $a_t$ satisfies that for all $x,\,y\in\mathbb{R}^n$ and $t>0$,
\begin{eqnarray}|a_t(x,\,y)|\le h_t(x,\,y)=t^{-n/s}h\Big(\frac{|x-y|}{t^{1/s}}\Big),\end{eqnarray}
where $s>0$ is a constant and $h$ is a positive, bounded and decreasing function such that for some constant $\eta>0$,
\begin{eqnarray}\lim_{r\rightarrow\infty}r^{n+\eta}h(r)=0.\end{eqnarray}
\end{definition}
\begin{assumption}\label{a1.0}
There exists an approximation to the identity $\{A_t\}_{t>0}$ such that the composite
operator $TA_t$  has an associated kernel $K_t$ in the sense of (1.6), and there
exists a positive constant $c_1$  such that for all $y\in \mathbb{R}^n$ and $t>0$,
$$\int_{|x-y|\geq c_1t^{\frac{1}{s}}}K(x,\,y)-K_t(x,\,y)|dx\lesssim 1.$$\end{assumption}
An $L^2(\mathbb{R}^n)$ bounded linear operator with kernel $K$ satisfying Assumption \ref{a1.0} is called a singular
integral operator with nonsmooth kernel, since $K$ does not enjoy smoothness in space
variables. Duong and McIntosh \cite{duongmc} proved that if $T$ is an $L^2(\mathbb{R}^n)$ bounded linear operator with
kernel $K$, and satisfies Assumption \ref{a1.0}, then $T$ is bounded from $L^1(\mathbb{R}^n)$ to $L^{1,\,\infty}(\mathbb{R}^n)$.
To consider the weighted estimates with $A_p(\mathbb{R}^n)$ boundedness of singular integral operators with non-smooth kernel, Martell \cite{mar} introduced the following assumptions.
\begin{assumption}\label{a1.1} There exists an approximation to the identity $\{D_t\}_{t>0}$ such that the composite
operator $D_tT$  has an associated kernel $K^t$ in the sense of (1.6), and there exist  positive
constants $c_2$ and $\alpha\in (0,\,1]$,  such that for all $t>0$ and $x,\,y\in\mathbb{R}^n$ with $|x-y|\geq c_2t^{\frac{1}{s}}$,
\begin{eqnarray*}&&|K(x,\,y)-K^t(x,\,y)|\lesssim\frac{t^{\alpha/s}}{|x-y|^{n+\alpha}}.
\end{eqnarray*}
\end{assumption}

Martell \cite{mar}
proved that if $T$ is an $L^2(\mathbb{R}^n)$ bounded linear operator, satisfies Assumption \ref{a1.0} and
Assumption \ref{a1.1}, then for any $p\in (1,\,\infty)$ and $w\in A_p(\mathbb{R}^n)$, $T$ is bounded on $L^p(\mathbb{R}^n,\,w)$.
The first purpose of this paper is to establish the endpoint vector-valued estimates for the corresponding grand maximal operator of singular integral operators with nonsmooth kernels. Our main result can be stated as follows.
\begin{theorem}\label{t1.3}Let T be an $L^2(\mathbb{R}^n)$ bounded linear operator with kernel $K$ as in
(1.5). Suppose that T satisfies Assumption \ref{a1.0} and   Assumption \ref{a1.1}. Then for each $\lambda>0$,
\begin{eqnarray}&&\big|\{x\in\mathbb{R}^n:\|\{Tf_k(x)\}\|_{l^q}+\|\{\mathcal{M}_{T}f_k(x)\}\|_{l^q}>\lambda\}\big|\\
&&\quad\lesssim \int_{\mathbb{R}^n}\frac{\|\{f_k(x)\}\|_{l^q}}{\lambda}\log\big(1+\frac{\|\{f_k(x)\}\|_{l^q}}{\lambda}\big)dx.\nonumber\end{eqnarray}
If we further assume that the kernels $\{K^t\}_{t>0}$ in  Assumption \ref{a1.1} also satisfy that for all $t >0$ and $x,\,y\in\mathbb{R}^n$ with
$|x-y|\leq c_2t^{\frac{1}{s}}$,
\begin{eqnarray}|K^t(x,\,y)|\lesssim t^{-\frac{n}{s}},\end{eqnarray} then (1.8) is also true for $T^*$, here and in the following, $T^*$ is the maximal singular integral operator defined by $$T^*f(x)=\sup_{\epsilon>0}|T_{\epsilon}f(x)|,$$
with $$T_{\epsilon}f(x)=\int_{|x-y|>\epsilon}K(x,\,y)f(y)dy.
$$
\end{theorem}
As a consequence of Theorem \ref{t1.3} and Theorem \ref{t1.1}, we obtain the following weighted vector-valued bounds for $T$ and $T^*$.

\begin{corollary}\label{t1.4}
Let T be an $L^2(\mathbb{R}^n)$ bounded linear operator with kernel $K$  in the sense of (1.9). Suppose that $T$ satisfies Assumption \ref{a1.0} and   Assumption \ref{a1.1}. Then
for $p\in (1,\,\infty)$ and $w\in A_{p}(\mathbb{R}^n)$,
\begin{eqnarray}&&\big\|\{Tf_k\}\big\|_{L^p(l^q;\,\mathbb{R}^n,w)}\lesssim_{n,\,p} [w]_{A_p}^{\frac{1}{p}}\big([\sigma]_{A_{\infty}}^{\frac{1}{p}}+[w]_{A_{\infty}}^{\frac{1}{p'}}\big)[\sigma]_{A_{\infty}}
\|\{f_k\}\|_{L^p(l^q;\,\mathbb{R}^n,\,w)},\end{eqnarray}
with $\sigma=w^{-\frac{1}{p-1}}$. Moreover, if the kernels $\{K^t\}_{t>0}$ in  Assumption \ref{a1.1}  satisfy (1.9), then the weighted estimate (1.10) also holds for $T^*$.
\end{corollary}
\begin{remark}We do not know if the weighted bound in (1.10) is sharp.
\end{remark}
In what follows, $C$ always denotes a
positive constant that is independent of the main parameters
involved but whose value may differ from line to line. We use the
symbol $A\lesssim B$ to denote that there exists a positive constant
$C$ such that $A\le CB$.  Constant with subscript such as $C_1$,
does not change in different occurrences. For any set $E\subset\mathbb{R}^n$,
$\chi_E$ denotes its characteristic function.  For a cube
$Q\subset\mathbb{R}^n$ and $\lambda\in(0,\,\infty)$, we use $\ell(Q)$ (${\rm diam}Q$) to denote the side length (diamter) of $Q$, and
$\lambda Q$ to denote the cube with the same center as $Q$ and whose
side length is $\lambda$ times that of $Q$. For $x\in\mathbb{R}^n$ and $r>0$, $B(x,\,r)$ denotes the ball centered at $x$ and having radius $r$.
\section{Proof of Theorem \ref{t1.3}}
We begin with some preliminary lemmas.
\begin{lemma}\label{l2.1}
Let $q,\,p_0\in (1,\,\infty)$, $\varrho\in [0,\,\infty)$ and $S$ be a sublinear operator. Suppose that
$$\|\{Sf_k\}\|_{L^{p_0}(l^q;\,\mathbb{R}^n)}\lesssim \|\{f_k\}\|_{L^{p_0}(l^q;\,\mathbb{R}^n)},$$
and for all $\lambda>0$,
$$\big|\{x\in\mathbb{R}^n:\,\|\{Sf_k(x)\}\|_{l^q}>\lambda\}\big|\lesssim \int_{\mathbb{R}^n}\frac{\|\{f_k\}\|_{l^q}}{\lambda}
\log ^{\varrho}\Big(1+\frac{\|\{f_k\}\|_{l^q}}{\lambda}\Big)dx.$$
Then for cubes $Q_2\subset Q_1\subset \mathbb{R}^n$,
$$\frac{1}{|Q_1|}\int_{Q_1}\big\|\{S(f_k\chi_{Q_2})(x)\}\big\|_{l^q}dx\lesssim \big\|\|\{f_k\}\|_{l^q}\big\|_{L(\log L)^{\varrho+1},\,Q_2},$$here and in the following,
for $\beta\in [0,\,\infty)$,
$$\|f\|_{L(\log L)^{\beta},\,Q}=\inf\Big\{\lambda>0:\,\frac{1}{|Q|}\int_{Q}\frac{|f(y)|}{\lambda}\log^{\beta}\Big(1+\frac{|f(y)|}{\lambda}\Big)dy\leq 1\Big\}.$$
\end{lemma}
\begin{proof} Lemma \ref{l2.1} is a generalization of Lemma 3.1 in \cite{huyang}. Their proofs are very similar. By homogeneity, we may assume that $\big\|\|\{f_k\}\|_{l^q}\big\|_{L(\log L)^{\varrho+1},\,Q_2}=1$, which implies that
$$\int_{Q_2}\|\{f_k(x)\}\|_{l^q}
\log ^{\varrho+1}\big(1+\|\{f_k(x)\}\|_{l^q}\big)dx\leq |Q_2|.$$
For each fixed $\lambda > 0$, set $\Omega_{\lambda}=\big\{x\in\mathbb{R}^n:\, \|\{f_k(x)\}\|_{l^q} >\lambda^{\frac{p_0-1}{2p_0}}\big\}$. Decompose $f_k$ as
$$f_k(x)= f_k(x)\chi_{\Omega_{\lambda}}(x)+f_k(x)\chi_{\mathbb{R}^n\backslash\Omega_{\lambda}}(x)=f_k^1(x)+f_k^2(x).$$
It is obvious that $\|\{f_k^2\}\|_{L^{\infty}(l^q;\,\mathbb{R}^n)}\leq \lambda^{\frac{p_0-1}{2p_0}}.$ A trivial computation leads to that
\begin{eqnarray*}
&&\int_{1}^{\infty}\big|\{x\in \mathbb{R}^n:\,\|\{S(f_k^2\chi_{Q_2})(x)\}\|_{l^q}>\lambda/2\}\big|d\lambda\\
&&\quad\lesssim\int_1^{\infty}\int_{Q_2}\big\|\{f_k^2(x)\}\big\|_{l^q}^{p_0}dx\lambda^{-p_0}d\lambda\\
&&\quad\lesssim\int_{Q_2}\big\|\{f_k^2(x)\}\big\|_{l^q}dx\int_1^{\infty}\lambda^{-p_0+\frac{(p_0-1)^2}{2p_0}}d\lambda\lesssim|Q_2|.\\
\end{eqnarray*}
On the other hand,
\begin{eqnarray*}
&&\int_{1}^{\infty}\big|\{x\in \mathbb{R}^n:\,\|\{S(f_k^1\chi_{Q_2})(x)\}\|_{l^q}>\lambda/2\}\big|d\lambda\\
&&\quad\lesssim\int_1^{\infty}\int_{Q_2}\big\|\{f_k^1(x)\}\big\|_{l^q}\log^{\varrho}\big(1+\big\|\{f_k^1(x)\}\big\|_{l^q}\big)dx\lambda^{-1}d\lambda\\
&&\quad\lesssim\int_{Q_2}\big\|\{f_k^1(x)\}\big\|_{l^q}\log^{\varrho}\big(1+\big\|\{f_k^1(x)\}\big\|_{l^q}\big)
\int_1^{\|\{f_k(x)\}\|_{l^q}^{\frac{2p_0}{p_0-1}}}\frac{1}{\lambda}d\lambda dx\\
&&\quad\lesssim \int_{Q_2}\big\|\{f_k^1(x)\}\big\|_{l^q}\log^{\varrho+1}\big(1+\big\|\{f_k^1(x)\}\big\|_{l^q}\big)dx.
\end{eqnarray*}
Combining the estimates above then yields
\begin{eqnarray*}&&\int_{0}^{\infty}\big|\{x\in Q_1:\,\|\{S(f_k\chi_{Q_2})(x)\}\|_{l^q}>\lambda\}\big|d\lambda\\
&&\quad\lesssim\int_{0}^{1}\big|\{x\in Q_1:\,\|\{S(f_k\chi_{Q_2})(x)\}\|_{l^q}>\lambda\}\big|d\lambda\\
&&\qquad+\int_{1}^{\infty}\big|\{x\in \mathbb{R}^n:\,\|\{S(f_k^1\chi_{Q_2})(x)\}\|_{l^q}>\lambda/2\}\big|d\lambda\\
&&\qquad+\int_{1}^{\infty}\big|\{x\in \mathbb{R}^n:\,\|\{S(f_k^2\chi_{Q_2})(x)\}\|_{l^q}>\lambda/2\}\big|d\lambda\\
&&\quad\lesssim|Q_1|.
\end{eqnarray*}
This  completes the proof of Lemma \ref{l2.1}.
\end{proof}
Recall that  the standard dyadic grid in $\mathbb{R}^n$ consists of all cubes of the form $$2^{-k}([0,\,1)^n+j),\,k\in \mathbb{Z},\,\,j\in\mathbb{Z}^n.$$
Denote the standard grid by $\mathcal{D}$. For a fixed cube $Q$, denote by $\mathcal{D}(Q)$ the set of dyadic cubes with respect to $Q$, that is, the cubes from $\mathcal{D}(Q)$ are formed by repeating subdivision of $Q$ and each of descendants into $2^n$ congruent subcubes.

As usual, by a general dyadic grid $\mathscr{D}$,  we mean a collection of cube with the following properties: (i) for any cube $Q\in \mathscr{D}$, it side length $\ell(Q)$ is of the form $2^k$ for some $k\in \mathbb{Z}$; (ii) for any cubes $Q_1,\,Q_2\in \mathscr{D}$, $Q_1\cap Q_2\in\{Q_1,\,Q_2,\,\emptyset\}$; (iii) for each $k\in \mathbb{Z}$, the cubes of side length $2^k$ form a partition of $\mathbb{R}^n$. By the one-third trick, (see \cite[Lemma 2.5]{hlp}), there exist dyadic grids $\mathscr{D}_1,\,\dots,\,\mathscr{D}_{3^n}$, such that for each cube $Q\subset\mathbb{R}^n$, there exists a cube $I\in\mathscr{D}_j$ for some $j$, $Q\subset I$ and $\ell(Q)\approx \ell(I)$.

Let $\{D_t\}_{t>0}$ be an approximation to the identity. Associated with $\{D_t\}_{t>0}$, define the sharp maximal operator $M_{D}^{\sharp}$ by
$$M_{D}^{\sharp}f(x)=\sup_{B\ni x}\frac{1}{|B|}\int_{B}|f(y)-D_{t_B}f(y)|dy,\,\,f\in L^p(\mathbb{R}^n),\,\,p\in [1,\,\infty)
$$
with $t_B=r_B^{s}$ and $s$ the constant appeared in (1.6),  the supremum is taken over all balls in $\mathbb{R}^n$. This operator was introduced by Martell \cite{mar} and plays an important role in the weighted estimates for singular integral operators with non-smooth kernels. Let $q\in (1,\,\infty)$,
$\{f_k\}\subset L^{p_0}(\mathbb{R}^n)$ for some $p_0\in [1,\,\infty]$,
set
$$M_{D}^{\sharp}(\{f_k\})(x)=\sup_{B\ni x}\frac{1}{|B|}\int_{B}\big\|\{|f_k(y)-D_{t_B}f_k(y)|\}\big\|_{l^q}dy.
$$
\begin{lemma}\label{l2.2}
Let $\lambda>0$, $\{f_k\}\subset L^1(\mathbb{R}^n)$ with compact support, $B\subset \mathbb{R}^n$ be a cube such that there
exists $x_0\in B$ with $M(\|\{f_k\}\|_{l^q})(x_0)<\lambda$. Then, for every $\zeta\in (0,\,1)$, we can find
$\gamma> 0$ (independent of $\lambda$, $B$, $f$, $x_0$), such that
\begin{eqnarray}|\{x\in B:\,M(\|\{f_k\}\|_{l^q})(x)>A\lambda,\,M^{\sharp}_D(\{f_k\})(x)\leq \gamma \lambda\}|\leq \zeta|B|,
\end{eqnarray}where $A>1$ is a fixed constant which only depends on  the approximation
of the identity $\{D_t\}_{t>0}$.
\end{lemma}
\begin{proof} Let $A\in (1,\,\infty)$  be a constant which will be chosen later. For $\lambda>0$, set
$$E_{\lambda}=\{x\in B:\, M\big(\|\{f_k\}\|_{l^q}\big)(x)>A\lambda,\,M_{D}^{\sharp}(\{f_k\})(x)\le \gamma\lambda\}.
$$
We assume that there exists $x_E\in E_{\lambda}$, for otherwise there is nothing to prove.
As in the proof of Proposition 4.1 in \cite{mar} (see also the proof of Lemma 2.6 of \cite{huyang}), we can verify that for each $x\in E_{\lambda}$ and  $\widetilde{A}=2^{-2n}A$,
$$M(\|\{f_k\}\|_{l^q}\chi_{4B})(x)>\widetilde{A}\lambda.
$$
Now let $t=r_{16B}^{s}.$ For $y\in 4B$, write
$$|D_t(f_k\chi_{16B})(y)|\leq \int_{16B}|h_t(y,\,z)f_k(z)|dz.$$
By Minkowski's inequality, we deduce that
$$\big\|\{|D_{t}(f_k\chi_{16B})(y)|\}\big\|_{l^q}\le \int_{16B}|h_t(y,\,z)|\|\{f_k(z)\}\|_{l^q}dz\lesssim M(\|\{f_k\}\|_{l^q})(x_0),$$
since $h$ is bounded on $[0,\,\infty)$. Also, we have that for $y\in 4B$,
\begin{eqnarray*}\big\|\{|D_{t}(f_k\chi_{\mathbb{R}^n\backslash 16B})(y)|\}\big\|_{l^q}&\le &\sum_{l=4}^{\infty}\int_{2^{l+1}B\backslash 2^{l}B}|h_t(y,\,z)|\|\{f_k(z)\}\|_{l^q}dz\\
&\lesssim&\sum_{l=4}^{\infty}\frac{1}{|B|}\int_{2^{l+1}B\backslash 2^{l}B}h(2^{l+4})\|\{f_k(z)\}\|_{l^q}dz\\
&\lesssim &M(\|\{f_k\}\|_{l^q})(x_0).\end{eqnarray*}
This, in turn implies that for all $y\in\mathbb{R}^n$,
$$M\big(\|\{(D_{t}f_k)\chi_{4B}\}\|_{l^q}\big)(x)\lesssim M(\|\{f_k\}\|_{l^q})(x_0)\le C_1\lambda,$$
with $C_1>0$ a constant. Therefore, for each $x\in E_{\lambda}$,
\begin{eqnarray*}
M\big(\|\{f_k\chi_{4B}\}\|_{l^q}\big)(x)&\leq &M\big(\|\{(f_k-D_tf_k)\chi_{4B}\}\|_{l^q}\big)(x)+M\big(\|\{(D_tf_k)\chi_{4B}\}\|_{l^q}\big)(x)\\
&\le&M\big(\|\{(f_k-D_tf_k)\chi_{4B}\}\|_{l^q}\big)(x)+C_1\lambda.
\end{eqnarray*}
We choose $A>1$ such that $\widetilde{A}=C_1+1$. It then follows that
$$E_{\lambda}\subset\{x\in B:\, M\big(\|\{(f_k-D_tf_k)\chi_{4B}\}\|_{l^q}\big)(x)>\lambda\}.$$
This, via the weak type $(1,\,1)$ estimate of $M$, tells us that
$$
|E_{\lambda}|\le C_2\lambda^{-1}\int_{4B}\|f_k(y)-D_tf_k(y)\}\|_{l^q}dy\le  C_216^n\lambda^{-1}|B|M_{D}^{\sharp}(\{f_k\})(x_E)\leq C_216^n\gamma|B|.
$$
For each $\zeta\in (0,\,1)$, let $\gamma=\zeta(2C_216^n)^{-1}$. The inequality (2.1) holds for this $\gamma$.\end{proof}
As in the proof of the Fefferman-Stein inequality (see \cite[pp 150-151]{gra}, or the proof of Theorem 2.2 in \cite{huyang}), we can deduce from Lemma \ref{l2.2} that
\begin{corollary}\label{l2.3}
Let $\Phi$ be an increasing function on $[0,\,\infty)$ satisfying that
$$\Phi(2t)\leq C\Phi(t),\,\,\,t\in[0,\,\infty).$$ $\{D_t\}_{t>0}$ be an approximation to the identity
as in Definition 1.4. Let $\{f_k\}$ be a sequence of functions such that for any $R>0$,
$$\sup_{0<\lambda<R}\Phi(\lambda)|\{x\in\mathbb{R}^n:\, M(\|\{f_k\}\|_{l^q})(x)>\lambda\}|<\infty.$$
Then
$$\sup_{\lambda>0}\Phi(\lambda)|\{x\in\mathbb{R}^n:M(\|\{f_k\}\|_{l^q})(x)>\lambda\}|\lesssim
\sup_{\lambda>0}\Phi(\lambda)|\{x\in\mathbb{R}^n:M_{D}^{\sharp}(\{f_k\})(x)>\lambda\}|.$$

\end{corollary}
\begin{lemma}\label{l2.6}Let T be an $L^2(\mathbb{R}^n)$ bounded linear operator with kernel $K$ as in
(1.5). Suppose that T satisfies Assumption \ref{a1.0}. Then for any $q\in (1,\,\infty)$, $T$ is bounded from $L^1(l^q;\,\mathbb{R}^n)$ to $L^{1,\,\infty}(l^q;\,\mathbb{R}^n)$
\end{lemma}
\begin{proof}
We only consider the case $c_1=2$. The other cases can be treated in the same way. For  $\lambda>0$, by applying the Calder\'on-Zygmund decomposition to $\|\{f_k\}\|_{l^q}$ at level $\lambda$,
 we obtain a sequence of cubes $\{Q_l\}$ with disjoint interiors, such that
$$\lambda<\frac{1}{|Q_l|}\int_{Q_l}\big\|\{f_k(x)\}\big\|_{l^q}{\rm d}x\lesssim \lambda,$$
and  $\|\{f_k(x)\}\|_{l^q}\lesssim \lambda$ for a. e. $x\in\mathbb{R}^n\backslash \cup_{l}Q_l$. For each  fixed $k$, set
$$f_{k}^1(x)=f_k(x)\chi_{\mathbb{R}^n\backslash(\cup_lQ_l)}(x),$$
$$f_k^{2}(x)=\sum_{l}A_{t_{Q_l}}b_{k,\,l}(x),\,\,
f_{k}^{3}(x)=\sum_{l}\big(b_{k,\,l}(x)-A_{t_{Q_l}}b_{k,\,l}(x)\big)\chi_{Q_l}(x),$$
with $b_{k,\,l}(y)=f_k(y)\chi_{Q_l}(y)$, $t_{Q_l}=\{\ell(Q_l)\}^{s}$.
By the fact that
$ \big\|\big\{f_{k}^1\big\}\big\|_{L^{\infty}(l^{q};\,\mathbb{R}^n)}\lesssim \lambda,$ we deduce that
$$\big\|\{f_{k}^{1}\}\big\|_{L^{q}(l^{q};\,\mathbb{R}^n)}^q\lesssim \lambda^{q-1} \big\|\{f_{k}\}\big\|_{L^1(l^{q};\,\mathbb{R}^n)}.$$
Recalling that $T$ is bounded on $L^q(\mathbb{R}^n)$, we have  that
\begin{eqnarray}\Big|\Big\{x\in\mathbb{R}^n:\big\|\big\{Tf_k^1(x)\big\}\big\|_{l^q}>\lambda/3\Big\}\Big|\lesssim \lambda^{-q}\big\|\{f_{k}^{1}\}\big\|_{L^{q}(l^{q};\,\mathbb{R}^n)}^q\lesssim \lambda^{-1}\big\|\{f_{k}\}\big\|_{L^1(l^{q};\,\mathbb{R}^n)}.
\end{eqnarray}
On the other hand, we  get from (1.5) and (1.6) that
\begin{eqnarray*}
\int_{\mathbb{R}^n}\big|v_k(y)A_{t_{Q_l}}b_{k,\,l}(y)\big|{\rm d}y&\le &\int_{Q_l}|b_{k,\,l}(z)|\int_{\mathbb{R}^n}h_{t_{Q_l}}(z,\,y)|v_k(z)|{\rm d}z dy\\
&\lesssim &\int_{Q_{l}}|b_{k,\,l}(z)|{\rm d}z\inf_{y\in Q_l}Mv_k(y).
\end{eqnarray*}
A straightforward computation involving  Minkowski's inequality  gives us that
$$\Big(\sum_{k}\|b_ {k,\,l}\|^{q}_{L^1(\mathbb{R}^n)}\Big)^{1/q}\le \int_{Q_l}\Big(\sum_k|f_ k(y)|^{q}\Big)^{1/q}\,dy\lesssim \lambda |Q_l|.
$$
Therefore, by  Minkowski's inequality and the vector-valued inequality of the Hardy-Littlewood maximal operator $M$ (see \cite{fes}),
\begin{eqnarray*}
&&\Big\|\Big(\sum_k\Big|\sum_lA_{t_{Q_ l}} b_ {k,\,l}\Big|^{q}\Big)^{1/q}\Big\|_{L^{q}(\mathbb{R}^n)}\\
&&\quad\le\sup_{\|\{v_k\}\|_{L^{q'}(l^{q'};\,\mathbb{R}^n)}\leq 1}
\sum_k\sum_l\int_{\mathbb{R}^n}\big|v_k(y)A_{t_{Q_l}}b_{k,\,l}(y)\big|dy\\
&&\quad\lesssim\sup_{\|\{v_k\}\|_{L^{p_m'}(l^{r_m'};\,\mathbb{R}^n)}\leq 1}
\sum_k\sum_l\int_{Q_l}|b_ {k,\,l}(z)|{\rm d}z\inf_{y\in Q_l}Mv_k(y)\\
&&\quad\lesssim\sup_{\|\{v_k\}\|_{L^{q'}(l^{q'};\,\mathbb{R}^n)}\leq 1}
\sum_l\Big\{\sum_k\Big(\int_{Q_l}|b_{k,\,l}(z)|dz\Big)^{q}\Big\}^{\frac{1}{q}}\inf_{y\in Q_m^l}\big\|\big\{Mv_k(y)\big\}\big\|_{l^{q'}}\\
&&\quad\lesssim\sup_{\|\{v_k\}\|_{L^{q'}(l^{q'};\,\mathbb{R}^n)}\leq 1}
\sum_l\int_{Q_l}\big\|\big\{b_{k,\,l}(z)\big\}\big\|_{l^{q}}dz\inf_{y\in Q_l}\big\|\big\{Mv_k(y)\big\}\big\|_{l^{q'}}\\
&&\quad\lesssim\lambda\sup_{\|\{v_k\}\|_{L^{q'}(l^{q'};\,\mathbb{R}^n)}\leq 1}\int_{\cup_{j}Q_j}\big\|\big\{Mv_k(y)\big\}\big\|_{l^{q'}}dy\\
&&\quad\lesssim\lambda^{\frac{q-1}{q}}\big\|\{f_{k}\}\big\|_{L^1(l^{q};\,\mathbb{R}^n)}^{\frac{1}{q}}.
\end{eqnarray*}
This, along with the fact that $T$ is bounded from $L^{q}(\mathbb{R}^n)$, leads to that
\begin{eqnarray}\big|\big\{x\in\mathbb{R}^n:\big\|\big\{Tf_k^2(x)\big\}\big\|_{l^q}>
\lambda/3\big\}\big|\lesssim\lambda^{-1}\big\|\{f_{k}\}\big\|_{L^1(l^{q};\,\mathbb{R}^n)}.
\end{eqnarray}

We turn our attention to $Tf_k^2$. Let $\Omega=\cup_{l}4nQ_l$. It is obvious that $|\Omega|\lesssim \lambda^{-1}\|\{f_k\}\|_{L^1(l^q;\,\mathbb{R}^n)}.$
For each $x\in\mathbb{R}^n\backslash\Omega$, write
$$
\big|Tf_k^3(x)\big|\leq  \sum_{l}\int_{\mathbb{R}^{n}}\big|K(x;y)-K_{A_{t_{Q_l}}}(x;y)\big||b_{k,l}(y)|dy
$$
Applying Minkowski's inequality twice, we obtain
\begin{eqnarray*}
\big\|\{Tf_k^3(x)\}\big\|_{l^q}\leq  \sum_{l}\int_{\mathbb{R}^{n}}\big|K(x;y)-K_{A_{t_{Q_l}}}(x;y)\big|\|\{b_{k,l}(y)\}\|_{l^q}dy
\end{eqnarray*}
Therefore,
\begin{eqnarray}
&&\big|\big\{x\in\mathbb{R}^n\backslash {\Omega}:\, \big\|\big\{Tf_k^3(x)\big\}\big\|_{l^q}>\lambda/3\big\}\big|\\
&&\quad\lesssim\lambda^{-1}\sum_l\int_{\mathbb{R}^n}\int_{\mathbb{R}^n\backslash 4nQ_l}\big|K(x;y)-K_{A_{t_{Q_l}}}(x;y)\big|dx\|\{b_{k,l}(y)\}\|_{l^q}dy\nonumber\\
&&\quad\lesssim\lambda^{-1}\|\{f_k\}\|_{L^1(l^q;\,\mathbb{R}^n)}.\nonumber
\end{eqnarray}
Combining the inequalities  (2.2)-(2.4) leads to our conclusion.\end{proof}
\begin{lemma}\label{l2.4}Let $T$ be the singular integral operator in Theorem \ref{t1.2}, then for each $N\in\mathbb{R}^n$ and functions $\{f_k\}_{k=1}^N\subset L^{p_0}(\mathbb{R}^n)$  for some $p_0\in [1,\,\infty)$,
$$M^{\sharp}_{D}(\{Tf_k\})(x)\lesssim M_{L\log L}(\|\{f_k\}\|_{l^q})(x)+\|\{Mf_k(x)\}\|_{l^q}.$$
\end{lemma}
\begin{proof}Let $x\in \mathbb{R}^n$, $B$ be a ball containing $x$ and $t_B=r_B^{s}$. Write
$$\frac{1}{|B|}\int_B\big\|\{|Tf_k(y)-D_{t_B}Tf_k(y)|\}\big\|_{l^q}dy\leq{\rm }E_1+{\rm E}_2+{\rm E_3},$$
with
$${\rm E}_1=\frac{1}{|B|}\int_B\|\{T(f_k\chi_{4B})(y)\}\|_{l^q}dy,$$
$${\rm E}_2=\frac{1}{|B|}\int_B\|\{D_{t_B}T(f\chi_{4B})(y)\}\|_{l^q}dy,$$
and
$${\rm E}_3=\frac{1}{|B|}\int_B\big\|\{|T(f_k\chi_{\mathbb{R}^n\backslash 4B})(y)-D_{t_B}T(f_k\chi_{\mathbb{R}^n\backslash 4B})(y)|\}\big\|_{l^q}dy.
$$
Recall that $T$ is bounded on $L^q(\mathbb{R}^n)$ (and so is bounded on $L^q(l^q;\,\mathbb{R}^n)$). Thus by Lemma \ref{l2.1} and Lemma \ref{l2.6},
$${\rm E}_1\lesssim \big\|\|\{f_k\}\|_{l^q}\big\|_{L\log L,\,4B}\lesssim M_{L\log L}(\|\{f_k\}\|_{l^q})(x).$$
On the other hand, it follows from Minkowski's inequality that
$$\|\{D_{t_B}T(f_k\chi_{4B})(y)\}\|_{l^q}\lesssim \int_{\mathbb{R}^n}|h_{t_B}(y,\,z)|\|\{T(f_k\chi_{4B})(z)\}\|_{l^q}dz
$$
Let $${\rm F}_0=\int_{16B}|h_{t_B}(y,\,z)|\|\{T(f_k\chi_{4B})(z)\}\|_{l^q}dz$$
and for $j\in\mathbb{N}$,
$${\rm F}_j=\int_{2^{j+5}B\backslash 2^{j+4}B}|h_{t_B}(y,\,z)|\|\{T(f_k\chi_{4B})(z)\}\|_{l^q}dz.$$
By the estimate (1.7)  and Lemma \ref{l2.1}, we know that
$${\rm F}_0\leq \big\|\|\{f_k\}\|_{l^q}\big\|_{L\log L,\,4B},$$
and
$${\rm F}_j\leq \frac{1}{|B|}h(2^j)\int_{2^{j+5}B}\|\{T(f_k\chi_{4B})(z)\}\|_{l^q}dz\lesssim 2^{-\delta j}\big\|\|\{f_k\}\|_{l^q}\big\|_{L\log L,\,4B}.$$
This, in turn gives us that
$${\rm E}_2\lesssim \big\|\|\{f_k\}\|_{l^q}\big\|_{L\log L,\,4B}.$$
Finally, Assumption \ref{a1.1} tells us that for each $k$ and $y\in B$,
$$\big|T(f_k\chi_{\mathbb{R}^n\backslash 4B})(y)-D_{t_B}T(f_k\chi_{\mathbb{R}^n\backslash 4B}(y)\big|\lesssim Mf_k(x),
$$
which implies that
$${\rm E}_3\lesssim \|\{Mf_k(x)\}\|_{l^q}.$$
Combining the estimates for ${\rm E}_1$, ${\rm E}_2$ and ${\rm E}_3$ then leads to our desired conclusion.
\end{proof}

Let $\mathscr{D}$ be a dyadic grid. Associated with $\mathscr{D}$, define the sharp maximal function $M^{\sharp}_{\mathscr{D}}$ as
$$M^{\sharp}_{\mathscr{D}}f(x)=\sup_{Q\ni x\atop{Q\in\mathscr{D}}}\inf_{c\in\mathbb{C}}\frac{1}{|Q|}\int_{Q}|f(y)-c|dy.$$
For $\delta\in (0,\,1)$, let $ M_{\mathscr{D},\,\delta}^{\sharp}f(x)=\big[M^{\sharp}_{\mathscr{D}}(|f|^{\delta})(x)\big]^{1/\delta}.$
Repeating the argument in \cite[p. 153]{ste2}, we can verify that if $\Phi$ is a increasing function on $[0,\,\infty)$ which satisfies that
$$\Phi(2t)\leq C\Phi(t),\,t\in [0,\,\infty),$$ then
\begin{eqnarray}&&\sup_{\lambda>0}\Phi(\lambda)|\{x\in\mathbb{R}^n:|h(x)|>\lambda\}|\lesssim
\sup_{\lambda>0}\Phi(\lambda)|\{x\in\mathbb{R}^n:M_{\mathscr{D},\delta}^{\sharp}h(x)>\lambda\}|,\end{eqnarray}
provided that $\sup_{\lambda>0}\Phi(\lambda)|\{x\in\mathbb{R}^n:\,M_{\mathscr{D},\,\delta}h(x)>\lambda\}|<\infty$.
\begin{lemma}\label{l2.5}
Under the assumption of Theorem \ref{t1.2}, for each $\lambda>0$,
$$\big|\{x\in\mathbb{R}^n:\,\|\{MTf_k(x)\}\|_{l^q}>\lambda\}\big|\lesssim \int_{\mathbb{R}^n}\frac{\|\{f_k\}\|_{l^q}}{\lambda}
\log\Big(1+\frac{\|\{f_k\}\|_{l^q}}{\lambda}\Big)dx.$$
\end{lemma}
\begin{proof} By the well known one-third trick  (see \cite[Lemma 2.5]{hlp}), we only need to prove that, for each dyadic grid $\mathscr{D}$, the inequality
\begin{eqnarray}&&\big|\big\{x\in\mathbb{R}^n:\big\|\big\{M_{\mathscr{D}}(Tf_k)(x)\big\}\big\|_{l^q}>1\big\}\big|\\
&&\quad\lesssim \int_{\mathbb{R}^n}
\|\{f_k(x)\}\|_{l^q}\log\big(1+\|\{f_k(x)\}\|_{l^q}\big)dx.\nonumber
\end{eqnarray}holds when $\{f_k\}$ is finite. As in the proof of Lemma 8.1 in \cite{csmp}, we can very that
for each cube $Q\in\mathscr{D}$, $\delta\in (0,\,1)$,
\begin{eqnarray*}\inf_{c\in\mathbb{C}}\Big(\frac{1}{|Q|}\int_{Q}\Big|\|\{M_{\mathscr{D}}f_k(y)\}\|_{l^q}-c\Big|^{\delta}dy
\Big)^{\frac{1}{\delta}}&\lesssim & \Big(\frac{1}{|Q|}\int_{Q}\|\{M_{\mathscr{D}}(f_k\chi_Q)\}\|_{l^q}^{\delta}\Big)^{\frac{1}{\delta}}\nonumber\\
&\lesssim&\langle\|\{f_k\chi_Q\}\|_{l^q}\rangle_{Q},\nonumber
\end{eqnarray*}
where in the last inequality, we invoked the fact that $M_{\mathscr{D}}$ is bounded from $L^1(l^q;\,\mathbb{R}^n)$ to $L^{1,\,\infty}(l^q;\,\mathbb{R}^n)$.  This, in turn, implies that
\begin{eqnarray}M_{\mathscr{D},\,\delta}^{\sharp}\big(\|\{M_{\mathscr{D}}f_k\}\|_{l^q}\big)(x)\lesssim M_{\mathscr{D}}\big(\|\{f_k\}\|_{l^q}\big)(x).\end{eqnarray}
Let $\Phi(t)=t\log^{-1}(1+t^{-1})$. It follows from (2.5), (2.7), (2.1) and Lemma \ref{l2.4} that
\begin{eqnarray*}
&&\big|\{x\in\mathbb{R}^n:\,\|\{M_{\mathscr{D}}Tf_k(x)\}\|_{l^q}>1\}\big|\\
&&\quad\lesssim\sup_{t>0}\Phi(t)
\big|\{x\in\mathbb{R}^n:\,M_{\mathscr{D},\,\delta}^{\sharp}\big(\|\{M_{\mathscr{D}}Tf_k\}\|_{l^q}\big)(x)>t\}\big|\\
&&\quad\lesssim\sup_{t>0}\Phi(t)
\big|\{x\in\mathbb{R}^n:\,M\big(\|\{Tf_k\}\|_{l^q}\big)(x)>\lambda\}\big|\\
&&\quad\lesssim\sup_{t>0}\Phi(t)|\{x\in\mathbb{R}^n:\, M_{D}^{\sharp}(\{Tf_k\})(x)>t\}|\\
&&\quad\lesssim\sup_{t>0}\Phi(t)\big|\{x\in\mathbb{R}^n:\,M_{L\log L}(\|\{f_k\}\|_{l^q})(x)+\|\{Mf_k(x)\}\|_{l^q}>t\}\big|\\
&&\quad\lesssim \int_{\mathbb{R}^n}\|\{f_k(x)\}\|_{l^q}\log(1+\|\{f_k(x)\}\|_{l^q}\big)dx,
\end{eqnarray*}
where in the last inequality, we have invoked the fact that $M$ is bounded from $L^1(l^q;\,\mathbb{R}^n)$ to $L^{1,\,\infty}(l^q;\,\mathbb{R}^n)$ (see \cite{fes}).
This establish (2.6) and completes the proof of Lemma \ref{l2.5}.\end{proof}

{\it Proof of Theorem \ref{t1.3}}. Let $q\in (1,\,\infty)$. We know by Lemma \ref{l2.6} that $T$ is bounded from $L^1(l^q;\,\mathbb{R}^n)$ to $L^{1,\,\infty}(l^q;\,\mathbb{R}^n)$. On the other hand, it was proved in \cite{duongmc} (see also \cite{mar}) that under the assumption of Theorem \ref{t1.3},
$$T^*f(x)\leq MTf(x)+Mf(x).$$
Thus by Lemma \ref{l2.5}, for each $\lambda>0$,
$$\big|\{x\in\mathbb{R}^n:\|\{T^*f_k(x)\}\|_{l^q}>\lambda\}\big|\lesssim \int_{\mathbb{R}^n}\frac{\|\{f_k(x)\}\|_{l^q}}{\lambda}\log\big(1+\frac{\|\{f_k(x)\}\|_{l^q}}{\lambda}\big)dx.$$
Therefore, it suffices to consider $\mathcal{M}_T$ and $\mathcal{M}_{T^*}$. On the other hand, it was proved in that maximal operator $M_{L\log L}$ satisfies that
\begin{eqnarray*}&&\big|\big\{x\in\mathbb{R}^n:\big\|\big\{M_{L\log L}f_k(x)\big\}\big\|_{l^q}>\lambda\big\}\big|\\
&&\quad\lesssim \int_{\mathbb{R}^n}
\frac{\|\{f_k(x)\}\|_{l^q}}{\lambda}\log\Big(1+\frac{\|\{f_k(x)\}\|_{l^q}}{\lambda}\Big)dx.\nonumber
\end{eqnarray*}Thus, by Lemma \ref{l2.5}, our proof is now reduced to proving that  the inequalities
\begin{eqnarray}\mathcal{M}_Tf(x)\lesssim MTf(x)+M_{L\log L}f(x).\end{eqnarray}
and
\begin{eqnarray}\mathcal{M}_{T^*}f(x)\lesssim MTf(x)+M_{L\log L}f(x).\end{eqnarray}
hold. Without loss of generality, we assume that $c_2>1$.

Let $Q\subset \mathbb{R}^n$ be a cube and $x,\,\xi\in Q$.
Set $t_Q=\big(\frac{1}{c_2\sqrt{n}}\ell(Q)\big)^{s}$ and write
\begin{eqnarray*}
T(f\chi_{\mathbb{R}^n\backslash 3Q})(\xi)&=&D_{t_Q}Tf(\xi)-D_{t_Q}T(f\chi_{3Q})(\xi)\\
&&+\Big(T(f\chi_{\mathbb{R}^n\backslash 3Q})(\xi)-D_{t_Q}
T(f\chi_{\mathbb{R}^n\backslash 3Q})(\xi)\Big).
\end{eqnarray*}
A trivial computation involving (1.8) leads to that
\begin{eqnarray*}
|D_{t_Q}Tf(\xi)|&\lesssim&|Q|^{-1}\sum_{j=1}^{\infty}\int_{2^{j}nt_Q^{\frac{1}{s}}<|\xi-y|\leq 2^{j+1}nt_Q^{\frac{1}{s}}}h\Big(\frac{|\xi-y|}{t_Q^{\frac{1}{s}}}\Big)|Tf(y)|dy\\
&&+|Q|^{-1}\int_{|\xi-y|\leq 2nt_Q^{\frac{1}{s}}}|Tf(y)|dy\\
&\lesssim&|Q|^{-1}\sum_{j=1}^{\infty}\int_{2^{j-1}nt_Q^{\frac{1}{s}}<|x-y|\leq 2^{j+2}nt_Q^{\frac{1}{s}}}h\Big(\frac{|x-y|}{2t_Q^{\frac{1}{s}}}\Big)|Tf(y)|dy\\
&&+|Q|^{-1}\int_{|x-y|\leq 3nt_Q^{\frac{1}{s}}}|Tf(y)|dy\\
&\lesssim&MTf(x).
\end{eqnarray*}
On the other hand, it follows from Lemma \ref{l2.1} that
\begin{eqnarray*}
|D_{t_Q}T(f\chi_{3Q})(\xi)|&\lesssim&\frac{1}{|Q|}\sum_{j=1}^{\infty}\int_{2^{j-1}nt_Q^{\frac{1}{s}}<|x-y|\leq 2^{j+2}nt_Q^{\frac{1}{s}}}h\Big(\frac{|x-y|}{2t_Q^{\frac{1}{s}}}\Big)|T(f\chi_{3Q})(y)|dy\\
&&+|Q|^{-1}\int_{|x-y|\leq 3nt_Q^{\frac{1}{s}}}|T(f\chi_{3Q})(y)|dy\\
&\lesssim& M_{L\log L}f(x).
\end{eqnarray*}
Finally, Assumption \ref{a1.1} tells us that
\begin{eqnarray*}
\Big|T(f\chi_{\mathbb{R}^n\backslash 3Q})(\xi)-D_{t_Q}T(f\chi_{\mathbb{R}^n\backslash 3Q})(\xi)\Big|&\lesssim&
\int_{\mathbb{R}^n\backslash 3Q}\big|K(\xi,\,y)-K^{t_Q}(\xi,\,y)\big||f(y)|dy\\
&\lesssim&t_{Q}^{\frac{\alpha}{s}}\int_{\mathbb{R}^n\backslash 3Q}\frac{1}{|\xi-y|^{n+\alpha}}|f(y)|dy\\
&\lesssim &Mf(x).
\end{eqnarray*}
Combining the estimates above leads to (2.8).

It remains to prove (2.9). Let $x,\,\xi\in Q$. Observe that ${\rm supp}\,\chi_{\mathbb{R}^n\backslash 3Q}(y)\subset\{y:\, |y-x|\geq \ell(Q)\}.$
\begin{eqnarray}T^*(f\chi_{\mathbb{R}^n\backslash 3Q})(\xi)\leq |T(f\chi_{\mathbb{R}^n\backslash 3Q})(\xi)|+\sup_{\epsilon\geq \ell(Q)}|T_{\epsilon}f(\xi)|.\end{eqnarray}
Now let $\epsilon\geq\ell(Q)$. Write
\begin{eqnarray*}
T_{\epsilon}(f\chi_{\mathbb{R}^n\backslash 3Q})(\xi)&=&D_{(\epsilon/c_2)^{s}}Tf(\xi)-D_{(\epsilon/c_2)^{s}}T(f\chi_{3Q})(\xi)\\
&&+\Big(T_{\epsilon}(f\chi_{\mathbb{R}^n\backslash 3Q})(\xi)-D_{\epsilon^{s}}
T(f\chi_{\mathbb{R}^n\backslash 3Q})(\xi)\Big).
\end{eqnarray*}
As in the argument for $\mathcal{M}_T$, we can verify that
$$|D_{(\epsilon/c_2)^{s}}Tf(\xi)|\lesssim M(Tf)(x)$$ and
$$|D_{(\epsilon/c_2)^{s}}T(f\chi_{3Q})(\xi)|\lesssim M_{L\log L}f(x).$$
As in \cite{duongmc}, write
\begin{eqnarray*}
&&T_{\epsilon}(f\chi_{\mathbb{R}^n\backslash 3Q})(\xi)-D_{\epsilon^{s}}T(f\chi_{\mathbb{R}^n\backslash 3Q})(\xi)=\int_{|\xi-y|\leq \epsilon}K^{(\epsilon/c_2)^{s}}(\xi,\,y)f(y)\chi_{\mathbb{R}^n\backslash 3Q}(y)dy\\
&&\quad+\int_{|\xi-y|>\epsilon}\big(K(\xi,\,y)-K^{(\epsilon/c_2)^{s}}(\xi,\,y)\big)f(y)\chi_{\mathbb{R}^n\backslash 3Q}(y)dy.
\end{eqnarray*}
The fact that $K^{(\epsilon/c_2)^{s}}$ satisfies the size condition (1.9), implies that
$$\Big|\int_{|\xi-y|\leq \epsilon}K^{(\epsilon/c_2)^{s}}(\xi,\,y)f(y)dy\Big|\lesssim\epsilon^{-n}\int_{|\xi-y|<\epsilon}|f(y)|dy\lesssim Mf(x).
$$
On the other hand, by the Assumption \ref{a1.1}, we obtain that
\begin{eqnarray*}
&&\Big|\int_{|\xi-y|>\epsilon}\big(K(\xi,\,y)-K^{(\epsilon/c_2)^{s}}(\xi,\,y)\big)f(y)\chi_{\mathbb{R}^n\backslash 3Q}(y)dy\Big|\lesssim Mf(x).
\end{eqnarray*}
Therefore,
$$\sup_{\epsilon\geq \ell(Q)}|T_{\epsilon}(f\chi_{\mathbb{R}^n\backslash 3Q})(\xi)|\lesssim MTf(x)+M_{L\log L}f(x),$$
which, via the estimates (2.8) and (2.10), shows that
$$\mathcal{M}_{T^*}f(x)\lesssim MTf(x)+M_{L\log L}f(x).$$
This completes the proof of Theorem \ref{t1.3}.
\qed

\end{document}